\documentclass[reqno]{amsart}
\usepackage{amsmath, amsfonts, amssymb}
\usepackage[margin=3cm]{geometry}
\usepackage[dvipsnames, svgnames]{xcolor}
\usepackage[inline]{asymptote}
\usepackage{adjustbox, amssymb, enumitem, tikz-cd, multirow}
\usepackage{mark-macros}

\usepackage[pdfusetitle, pdfencoding=auto, psdextra, breaklinks=true]{hyperref}
\usepackage{xcolor}

\usepackage[nameinlink]{cleveref}

\renewcommand{\textbf}[1]{{\bfseries\boldmath #1}}
\newcommand{\vocab}[1]{\textbf{\textcolor{BrickRed}{\boldmath #1}}}

\definecolor{mylinkcolor}{rgb}{0.0,0.0,0.7}
\definecolor{myurlcolor}{rgb}{0.0,0.0,0.7}
\hypersetup{
 colorlinks,
 urlcolor=myurlcolor,
 citecolor=myurlcolor,
 linkcolor=mylinkcolor,
 breaklinks=true
}

\usepackage{thmtools}
\declaretheorem{theorem}[name=Theorem, numberwithin=section]
\declaretheorem{lemma}[name=Lemma, sibling=theorem]
[name=Claim, sibling=theorem]
[name=Proposition, sibling=theorem]
\declaretheorem{corollary}[name=Corollary, sibling=theorem]
\declaretheorem{definition}[style=definition, name=Definition, sibling=theorem]
[style=definition, name=Remark, sibling=theorem]
[style=definition, name=Example, sibling=theorem]
[style=definition, name=Open Problem, sibling=theorem]

\usepackage[backend=bibtex, style=alphabetic, sorting=nyt,
maxnames=20, maxalphanames=20, backref, url=true, doi=true]{biblatex}
\title{On $t$-intersecting Families of Spanning Trees}
\author{Pitchayut Saengrungkongka}

\begin{document}
\begin{abstract}
We prove that there exists a constant $c>0$ 
such that for all integers $2\leq t\leq cn$,
if $\calA$ is a collection of spanning trees in $K_n$
such that any two intersect at at least $t$ edges,
then $|\calA|\leq 2^tn^{n-t-2}$.
This bound is tight; the equality is achieved when $\calA$ 
is a collection of spanning trees containing 
a fixed $t$ disjoint edges.
This is an improvement of a result by Frankl, Hurlbert,
Ihringer, Kupavskii, Lindzey, Meagher, and Pantagi,
who proved such a result for $t=O\left(\frac n{\log n}\right)$.
\end{abstract}
\maketitle

\section{Introduction}
\subsection{Our Result}
A family $\calA$ of spanning trees of $K_n$ is said to be 
\vocab{$t$-intersecting} if and only if 
$T_1$ and $T_2$ intersect at at least $t$ edges
for all $T_1,T_2\in\calA$.
A natural question to ask is ``what is the maximum 
size of a $t$-intersecting family of spanning trees in $K_n$?''
We answer this question when $t\leq cn$ for 
an absolute constant $c$.
More specifically, we prove the following theorem.
\begin{theorem}
\label{thm:main}
There exists a constant $c>0$ such that 
for all nonnegative integers $t$ and $n$ such that 
$2\leq t\leq cn$, 
if $\calA$ is a $t$-intersecting family 
of spanning trees in $K_n$, then 
$$|\calA|\leq 2^tn^{n-t-2}.$$
\end{theorem}
This is a direct improvement of a previous result by 
Frankl, Hurlbert, Ihringer, Kupavskii, Lindzey, Meagher, and Pantagi
\cite[Thm.~1.4]{trees}, who proved the above theorem for 
$t=O\left(\tfrac n{\log n}\right)$.
We note that the bound $2^t n^{n-t-2}$ is tight:
if $F$ is a union of $t$-disjoint edges 
and $\calA$ is the family of all spanning trees 
containing all edges in $F$, then by \Cref{cor:trivial_case} (a), we have $|\calA| = 2^tn^{n-t-2}$

In the case of $t=1$ and $n$ is sufficiently large,
the maximum size of $t$-intersecting family 
is not $2n^{n-3}$.
Rather, the union of 
\begin{itemize}
\item all stars; and
\item the set of trees containing edge $e$
\end{itemize}
is $1$-intersecting, and has size $2n^{n-3}+n-2$.
The same paper also proves that this is tight for all 
sufficiently large $n$:
\begin{theorem}[{\cite[Thm.~1.5]{trees}}]
\label{thm:original}
If $n\geq 2^{19}$ and $\calA$ is a $1$-intersecting 
family of spanning trees in $K_n$,
then $|\calA|\leq 2n^{n-3}+n-2$.
\end{theorem}

\subsection{Erd\H os-Ko-Rado Combinatorics}
The question of finding the maximum size 
of $t$-intersecting families of spanning trees 
has many variants.
The simplest variant dates back to 
the Erd\H os-Ko-Rado theorem \cite{erdos_ko_rado},
which concerns a family of sets.
A family of sets $\calA$ is $t$-intersecting 
if $|X_1\cap X_2|\geq t$ for all $X_1,X_2\in\calA$.
The Erd\H os-Ko-Rado theorem states that
if $n\geq 2k$, then any $1$-intersecting family 
$\calA\subseteq \binom{[n]}k$ has size at most 
$\binom{n-1}{k-1}$.
This theorem was later generalized to 
$t$-intersecting families, culminating with 
the Ahlswede-Khachatrian theorem \cite{sets},
which determines the maximum size of 
$t$-intersecting family $\calA\subseteq \binom{[n]}k$
for all triples $(n,k,t)$.

One may ask a similar question on permutations.
Let $S_n$ be the symmetric group on $[n]$.
A family of permutations $\calA\subseteq S_n$
is $t$-intersecting if for any two $\sigma_1,\sigma_2\in\calA$,
we have $\sigma_1(i)=\sigma_2(i)$ 
for at least $t$ values of $i$.
Like above, a natural question is to determine the maximum 
size of $t$-intersecting family of permutations in $S_n$,
which was answered by Kupavskii \cite{permutations_better}
for $t<(1-\eps)n$ and $n>n_0(\eps)$.

\subsection{Our Techniques}
Our proof is based on the technique of spread approximation,
which was introduced in \cite{spread_approx}
to solve Erd\H os-Ko-Rado type problems in various settings.
The idea of spread approximation is to approximate 
a $t$-intersecting family of sets $\calA$
with another $t$-intersecting family of sets $\calS$,
each of whose elements has smaller size than elements in $\calA$.
This technique was used to prove \Cref{thm:original}
in \cite{trees}.

On the other hand, the spread approximation technique 
has been refined to handle the larger value of $t$
(compared to $n$) in \cite{permutations_better}.
There, they first prove the ``density boost'' statement
\cite[Thm.~12]{permutations_better}, 
by finding a set of size close to $t$ 
that is contained in many elements in $\calA$.
Then, they use density boost to obtain a more accurate 
spread approximation. 
Our proof of \Cref{thm:main} follows the same strategy 
of obtaining a density boost first.
However, the key difference is that we were able 
to prove density boost by a more direct argument,
not using another spread approximation as in 
\cite{permutations_better}.

\subsection{Outline}
\Cref{sec:density_boost} proves the density boost statement
that there exists a set of size close to $t$
and is contained in $e^{-\eps t}$-fraction 
of elements in $\calA$.
Then, \Cref{sec:spread_approximation} uses this 
to obtain a stronger spread approximation,
resulting in \Cref{lem:strong_spread_approx}.
\Cref{sec:counting_trees} starts off by recalling 
some basic enumerative facts about trees.
Then, its main goal is to prove \Cref{lem:counting_avoiding_trees},
which counts the number of trees that contains $F$
but avoids any edge in $T_0\setminus F$.
This section uses the Lopsided Lov\'asz Local Lemma.
Finally, \Cref{sec:main_proof} puts everything together 
and proves \Cref{thm:main}.

\subsection*{Acknowledgements}
This research was conducted at 
the University of Minnesota Duluth REU 
with support from Jane Street Capital, the NSF Grant 2409861,
and donations from Ray Sidney and Eric Wepsic.
We thank Colin Defant and Joe Gallian for such 
a wonderful opportunity.
We also thank Noah Kravitz
for helpful discussions and feedback on this paper.

\section{Preliminaries}
\subsection{Background and Notations}
\label{subsec:notation}
Let $K_n$ denote a complete graph with $n$ vertices.
A \vocab{forest} is a graph without cycles.
A \vocab{tree} is a connected forest.
A forest is said to be \vocab{spanning} in $K_n$
if it uses the same vertex set as $K_n$.
A \vocab{star} is a tree in which there is 
one vertex adjacent to all other vertices.

We let $[n] = \{1,2,\dots,n\}$ be the standard $n$-element set.
For any set $X$, we let 
\begin{itemize}
\item $2^X$ be the set of all subsets of $X$.
\item $\binom{X}k$ 
be the set of all subsets of $X$ of size $k$.
\item $\binom{X}{\leq k}$
be the set of all subsets of $X$ of size at most $k$.
\end{itemize}
Calligraphic letters ($\calA$, $\calB$, \dots)
will often be used to denote families of sets.
As in \cite{spread_approx}, we use the following notation in dealing with 
families of sets:
for set $X$ and families $\calF$, $\calS$,
we let 
\begin{align*}
\calF[X] &:= \{F\in\calF : X\subseteq F\} \\
\calF(X) &:= \{F\setminus X : F\in\calF, X\subseteq F\} \\
\calF[\calS] &:= \bigcup_{A\in\calS} \calF[A].
\end{align*}

For any family $\calU\subseteq \binom{[N]}n$,
we let 
\begin{equation}
\label{eq:d_i}
c_i(\calU) := \max_{S\in\binom{[N]}i} |\calU[S]|,
\quad 
d_i(\calU) := \frac{c_i(\calU)}{|\calU|}
\end{equation}
(For example, $c_0(\calU) = |\calU|$ and $d_0(\calU) = 1$.)

We will need the following classical estimate
of binomial coefficients.
\begin{lemma}
\label{lem:binom_bound}
For any integers $n\geq k\geq 0$,
we have $\binom nk \leq (en/k)^k$.
\end{lemma}
\begin{proof}
Note that $\binom nk \leq \frac{n^k}{k!}$
and $k!\leq \left(\tfrac ke\right)^k$ 
(by integrating $\log x$).
Combining these two estimates gives the lemma.
\end{proof}
\subsection{Spread Families}
\label{subsec:spread}
\begin{definition}
A family $\calA\subseteq 2^{[N]}$ is said to be 
\vocab{$r$-spread} if for all $S\subseteq [N]$, we have 
$$|\calA[S]| \leq r^{-|S|}|\calA|.$$
A family $\calA\subseteq 2^{[N]}$ is said to be 
\vocab{$(r,t)$-spread} if $\calA(T)$ is $r$-spread 
for all $t$-element subset $T\subseteq [N]$.
\end{definition}
Observe that if $\calU$ is $(r,t)$-spread, then 
$d_i(\calU) \leq r^{-(i-j)} d_j(\calU)$
for all indices $i,j$ such that $i>j$ and $j\leq t$.

Spread families behave nicely in terms of 
applying the probabilistic method.
In the following lemma, a \vocab{$p$-random subset}
$W\subseteq [N]$ is a random subset obtained by 
including each element in $[N]$ independently with 
probability $p$.
\begin{lemma}[Spread Lemma, {\cite[Prop.~5]{spread_lemma}}]
\label{lem:spread_lemma}
Let $\calA\subseteq \binom{[N]}{\leq k}$ be an $r$-spread family
and $W$ be a $(\beta\delta)$-random subset of $[N]$.
Then,
$$\mathbb P(\text{there exists } F\in\calA
\text{ such that }F\subseteq W) 
\geq 1 - k\left(\frac{2}{\log_2(r\delta)}\right)^\beta.$$
\end{lemma}

We record the following useful observation 
that will allow us to find spread sets.
\begin{lemma}[Finding Spread Family]
\label{lem:spread_family}
Let $\calU\subseteq 2^{[N]}$ be an $r$-spread family,
and let $\calA\subseteq \calU$ be subfamily.
Let $r'<r$.
Then, there exists $S\subseteq [N]$ 
such that 
\begin{itemize}
\item $\calA(S)\subseteq 2^{[N]\setminus S}$ is $r'$-spread.
\item $|S| \leq \dfrac{\log|\calU| - \log |\calA|}
{\log r - \log r'}$. 
\end{itemize}
\end{lemma}
\begin{proof}
Take $S$ such that $|\calA[S]| \cdot (r')^{|S|}$ is maximal.
We first check that $\calA(S)$ is $r'$-spread.
To do that, note that for any $X\subseteq [N]\setminus S$,
we have 
$$|\calA(S)[X]| = |\calA(S\cup X)| 
\leq (r')^{-|X|}|\calA(S)|$$
by using maximality between $S$ and $S\cup X$.

To verify the second condition, we note that
\begin{align*}
(r')^{-|S|}|\calA| &\leq |\calA[S]| \tag{maximality between 
$S$ and $\emptyset$} \\
&\leq |\calU[S]| 
\leq r^{-|S|} |\calU|,
\tag{spreadness of $|\calU|$}
\end{align*}
which verifies the second condition.
\end{proof}
\section{Density Boost}
\label{sec:density_boost}
In this section, we prove the density boost lemma.
This is essentially the same as 
\cite[Thm.~12]{permutations_better}.
There, it was proved by first obtaining an initial 
spread approximation.
We provide a more direct proof.
\begin{lemma}[Density Boost]
\label{lem:density_boost}
For any $\eps > 0$, there exists $\delta>0$ 
and $n_0=n_0(\eps)$
such that the following holds:

Let $\calU\subseteq \binom{[N]}n$ be an $(r,t)$-spread family,
and let $\calA\subseteq\calU$ be a $t$-intersecting family
such that $|\calA| \geq 0.001c_t(\calU)$.
Assume that $t > n^{1-\delta}/\delta$, $r>n^{0.99}$, and $n>n_0$.
Then, there exists a set $X\subseteq [N]$ such that 
$|X| = t - \lceil n^{1-\delta}\rceil$
and 
$$|\calA[X]| \geq  e^{-\eps t} |\calA|.$$
\end{lemma}
\begin{proof}
Select $\delta\in \left(0,\tfrac 16\right)$ such that $12\delta(1-\log\delta) < \eps$.
We select $S$ such that $d_{|S|}(\calU)^{-2\delta}
\cdot |\calA[S]|$ is maximal.
We first show that $|S| \leq (1+6\delta)t$.
To do this, assume for a contradiction that $|S|>(1+6\delta)t$.
Then, observe that 
\begin{align*}
0.001 d_t(\calU)\, |\calU| < |\calA| 
&\leq d_{|S|}(\calU)^{-2\delta}  |\calA[S]|
\tag{maximality condition for $S$ and $\emptyset$} \\
&\leq d_{|S|}(\calU)^{-2\delta}  |\calU[S]| \\
&\leq d_{|S|}(\calU)^{1-2\delta}  |\calU| 
\tag{definition \eqref{eq:d_i} of $d_i(\calU)$} \\
&\leq r^{-(1-2\delta)(|S|-t)} d_t(\calU)^{1-2\delta} 
|\calU| \tag{$(r,t)$-spreadness of $\calU$} \\
&< r^{-6\delta(1-2\delta)t} d_t(\calU)^{1-2\delta} |\calU|,
\tag{assumption that $|S|> (1+6\delta)t$}
\end{align*}
which implies that 
$$0.001 d_t(\calU)^{2\delta} \leq r^{-6\delta(1-2\delta)t}
\implies 0.001 d_t(\calU) \leq r^{-3(1-2\delta)t},
$$
By a straightforward averaging argument, 
$d_t(\calU) \geq n^{-t} \geq r^{-2t}$,
which is a contradiction.

Now, let $t'=t-\lceil n^{1-\delta}\rceil$,
and assume for contradiction that such $X$ does not exist.
We claim that there exists an element 
$F\in\calA$ such that $|F\cap S| < t'$.
To prove this, assume for contradiction that 
$|F\cap S| \geq t'$ for all $F\in\calA$.
We count in two ways 
the number
$$N = \left|\left\{(A,T) : A\in\calA, T\in 
\tbinom{A\cap S}{t'} \right\}\right|.$$
On one hand, we get that $N\geq |\calA|$.
On the other hand, we get that 
\begin{align*}
N\leq \sum_{X\in \binom{S}{t'}} |\calA[X]|
&\leq \binom{|S|}{t'} \cdot |\calA|\cdot e^{-\eps t} \\
&< \binom{(1+6\delta)t}{12\delta t} \cdot |\calA|\cdot e^{-\eps t} \\
&\leq \left(\frac{e(1+6\delta) t}{12\delta t}\right)^{12\delta t} 
\cdot |\calA|\cdot e^{-\eps t} 
\tag{\Cref{lem:binom_bound}}\\
&\leq \left(\frac{e}{\delta}\right)^{12\delta t} |\calA| \cdot e^{-\eps t} 
< |\calA|,
\end{align*}
from our choice of $\delta$.
This gives a contradiction.
Thus, $|F\cap S|<t'$ for some $F\in \calA$.

Observe that any 
$A\in\calA[S]$ must intersect $F$ at $t-t'\geq n^{1-\delta}$ elements 
outside $S$,
i.e., $T$ contains a subset $U\in \binom{F\setminus S}
{t-t'}$.
Summing across $t$ gives
\begin{align*}
|\calA[S]| &\leq \sum_{T\in\binom{F\setminus S}{t-t'}} 
|\calA[S\cup T]| \\
&\leq \binom{n-|S|}{n^{1-\delta}} 
\left(\frac{c_t(\calU)}{c_{t'}(\calU)}\right)^{2\delta} 
|\calA[S]| 
\tag{maximality between $S$ and $S\cup T$} \\
&\leq \binom{n}{n^{1-\delta}}
r^{-2\delta n^{1-\delta}} |\calA[S]| 
\tag{$(r,t)$-spreadness of $\calU$}\\
&\leq (en^{\delta})^{n^{1-\delta}}
r^{-2\delta n^{1-\delta}} |\calA[S]| 
\tag{\Cref{lem:binom_bound}}\\
&< |\calA[S]| \tag{$r>n^{0.99}$ and $n$ is sufficiently large}
\end{align*}
which is a contradiction.
\end{proof}
\section{Spread Approximation}
\label{sec:spread_approximation}
In this section, we use the density boost lemma 
(\Cref{lem:density_boost}) 
to obtain a spread approximation, 
which is an approximation of $\calA$ 
with a family of sets $\calS$ with lower uniformity.
\begin{lemma}
\label{lem:find_spread_set}
Assume the same setup as \Cref{lem:density_boost}
and take $\delta$ as in that lemma.
Then, there exists $n_0=n_0(\eps)$ such that 
whenever $n>n_0$, there exists $Y\subseteq [N]$
such that $|Y|\leq (1+\eps)t$ and $\calA(Y)$ is $(r/2)$-spread.
\end{lemma}
\begin{proof}
Take $X$ as in \Cref{lem:density_boost}.
Let $t' = |X| = t-\lceil n^{1-\delta}\rceil$.
By \Cref{lem:spread_family} on $\calA(X)\subseteq \calU(X)$,
we get that there exists $Y$ such that 
$\calA(Y)$ is $(r/2)$-spread and 
\begin{align*}
|Y| &\leq |X| + \frac{\log\calU(X) - \log\calA(X)}{
    \log r - \log(r/2) } \\
    &\leq t - n^{1-\delta}
    + \frac{\log c_{t'}(\calU)
    -\log(0.001 e^{-\eps t} c_{t'}(\calU))}{
        \log 2} \\
    &= t -  n^{1-\delta} + 
    \frac{\eps t + O(1)}{\log 2},
\end{align*}
which is less than $(1+\eps)t$ if $n$ is very large.
\end{proof}
\begin{lemma}
\label{lem:strong_spread_approx}
Let $\eps>0$ and take $\delta$ as in \Cref{lem:density_boost}.
Let $r>\max(20\eps t,\sqrt n)$, $t > \tfrac 1{\delta} n^{1-\delta}$,
and $n>n_0(\eps)$.
Let $\calU\subseteq \binom{[N]}n$ be an $(r,t)$-spread family,
and let $\calA\subseteq\calU$ be a $t$-intersecting family
such that $|\calA| \geq 0.001 c_t(\calU)$.
Then, there exists family $\calS \subseteq 
\binom{[N]}{\leq (1+\eps)t}$ and $\calA'\subseteq A$ such that 
\begin{enumerate}[label=(\alph*)]
\item $\calA\setminus \calA'\subseteq \calU[\calS]$.
\item $|\calA'| \leq 0.001 c_t(\calU)$.
\item $\calS$ is $t$-intersecting.
\end{enumerate}
\end{lemma}
\begin{proof}
We iteratively apply \Cref{lem:find_spread_set}.
Initialize $\calA_1=\calA$.
Then, for each $i$, find the set $Y_i$
such that $|Y_i|\leq (1+\eps)t$ and 
$\calA_i(Y_i)$ is $(r/2)$-spread.
Then, let $\calA_{i+1} = \calA_i\setminus\calA_i(Y_i)$.
The process stops at the $s$-step, where 
$|\calA_{s+1}| \leq 0.001c_t(\calU)$.
Let $\calS = \{Y_1,Y_2,\dots,Y_s\}$,
and we let $\calA' = \calA_{s+1}$.

Clearly, the (a) and (b) are satisfied.
We are now left with (c).
Assume for contradiction that $|Y_i\cap Y_j|\leq t-1$
for some $i<j$. Let $x=t-|Y_i\cap Y_j|$, and define 
\begin{align*}
\calB_i &= \{X\in\calA_i(Y_i) : |X\cap Y_j| \leq  
\lfloor x/2\rfloor\}  \\
\calB_j &=  \{X\in\calA_j(Y_j) : |X\cap Y_i| \leq \lfloor x/2\rfloor\}  
\end{align*}
We claim that $|\calB_i| \geq \tfrac 13|\calA_i(Y_i)|$.
To see this, note that 
\begin{align*}
|\calA_i(Y_i)| - |\calB_i|
&\leq \sum_{T\in \binom{Y_j\setminus Y_i}{\lceil x/2\rceil}} 
\calA_i(Y_i\cup T) \\
&\leq \binom{|Y_j\setminus Y_i|}{\lceil x/2\rceil} 
\left(\frac r2\right)^{-\lceil x/2\rceil}
|\calA_i(Y_i)| 
\\
&=  
\left(\frac{2e(\eps t + x)}x\right)^{\lceil x/2\rceil}
\left(\frac r2\right)^{-\lceil x/2\rceil} |\calA_i(Y_i)|
\tag{\Cref{lem:binom_bound}} 
\\
&\leq (10\eps t + 2e)^{\lceil x/2\rceil} 
\left(\frac r2\right)^{-\lceil x/2\rceil}
|\calA_i(Y_i)| \\
&\leq \tfrac 23 |\calA_i(Y_i)|.
\end{align*}
Thus, $\calB_i$ is $(r/6)$-spread.
Similarly, $\calB_j$ is $(r/6)$-spread.

Now, we apply spread lemma (\Cref{lem:spread_lemma}).
Randomly color elements in $[N]$ with $2$ colors,
where each number is independently assigned a color with 
probability $\tfrac 12$.
Let $U_i$ and $U_j$ be the set of numbers 
with each color. By \Cref{lem:spread_lemma}
(with $\delta = \frac {12}r$ and 
$\beta = \frac r{24}$), we have 
\begin{align*}
\mathbb P(\text{there exists }F_i'\in\calB_i 
\text{ such that } F_i'\subseteq U_i)
&\geq 1- 2^{-r/24} n \\
\mathbb P(\text{there exists }F_j'\in\calB_j 
\text{ such that } F_j'\subseteq U_j)
&\geq 1- 2^{-r/24} n.
\end{align*}
We can make $n$ large enough so that $r/24 > n^{0.99}/24 > \log_2(2n)$,
so $2^{-r/24}n < \tfrac 12$.
Thus, by union bound, there exists $F_i'\in\calB_i$
and $F_j'\in\calB_j$ such that $F_i'\cap F_j'=\emptyset$.
We have $F_i := F_i'\cup Y_i\in \calA$
and $F_j := F_j'\cup Y_j\in \calA$.
Then, note that
\begin{align*}
|F_i\cap F_j| &\leq  
|Y_i\cap F_j| + |F_i'\cap F_j| \\ 
&\leq |Y_i\cap Y_j| + |Y_i\cap F_j'|
+ |F_i'\cap Y_j| \\
&\leq |Y_i\cap Y_j|  + \tfrac x2 + \tfrac x2 < t,
\end{align*}
which violates the $t$-intersecting property of $\calA$.
\end{proof}

Once we obtained the spread approximation $\calS$,
we use a result due to \cite{partitions}
to analyze the structure of $\calA[\calS]$.
Before stating the theorem, we make the following definition:
a $t$-intersecting family $\calS$ is \vocab{trivial}
if every element in $\calS$ contain a fixed subset 
$F$ of size $t$.
\begin{theorem}[{\cite[Thm.~14]{partitions}}]
\label{thm:structure}
Let $\eps\in (0,1)$ and $N,r,q,t\geq 1$ be integers 
such that $\eps r \geq 24q$.
Let $\calA\subseteq 2^{[N]}$ be an $(r,t)$-spread family,
and let $\calS\subseteq \binom{[N]}{\leq q}$
be a nontrivial $t$-intersecting family.
Then, there exists a $t$-element set $T$ such that 
$|\calA[\calS]| \leq \eps |\calA[T]|$.
\end{theorem}
\section{Families of Spanning Trees}
\label{sec:counting_trees}
We now analyze the family $\calT_n$ of 
spanning tree of complete graph $K_n$.
We use the following formula for number of spanning trees
containing a fixed forest.
\begin{lemma}[{\cite[Lem.~6]{lll_trees}}]
\label{lem:counting_trees}
Let $F$ be a forest with $r$ connected component 
with size $q_1,\dots,q_r$.
Then, 
$$|\calT_n[F]| = q_1q_2\dots q_r n^{n-2-\sum_{i=1}^r 
(q_i-1)}.$$
\end{lemma}
Using this formula, it is straightforward 
to deduce the following corollaries.
\begin{corollary}
\label{cor:trivial_case}
Let $t\leq n/2$, and let $F$ be a forest 
in $K_n$ with $t$ edges.
\begin{enumerate}[label=(\alph*)]
\item If $F$ is a union of $t$ vertex-disjoint edges,
then $|\calT_n[F]| = 2^tn^{n-t-2}$.
\item Otherwise, we have $|\calT_n[F]| \leq 
\tfrac 34\cdot 2^tn^{n-t-2}$.
\end{enumerate}
\end{corollary}
\begin{proof}
\begin{enumerate}[label=(\alph*)]
\item Immediate from \Cref{lem:counting_trees}.
\item Let $a_1,\dots,a_m\geq 2$ be the sizes of 
connected components of $F$, not including isolated vertices.
Then, $\sum_{i=1}^m (a_i-1) = t$.
Consider the choice of $m$ and $a_1,\dots,a_m$
such that $\sum_{i=1}^m (a_i-1)=1$,
not all of $a_1,\dots,a_m$ are $2$,
and the product $a_1\cdots a_m$ is maximal.
If $m\leq t-2$, then at least one index, say $a_1$,
is at least $3$. Then, we replace 
$(a_1,\dots,a_m)$ by $(a_1-1,a_2,a_3,\dots,a_m,2)$
to get a higher product.
We do this until we reach $m=t-1$,
in which case $(a_1,\dots,a_m)$ is a permutation of 
$(3,2,2,\dots,2)$, so $a_1\cdots a_m \leq 3\cdot 2^{t-2}$,
and hence 
\[|\calT_n[F]| = a_1\cdots a_m n^{n-t-2} 
\leq 3\cdot 2^{t-2} \cdot n^{n-t-2}.
\qedhere\]
\end{enumerate}
\end{proof}
\begin{corollary}[{\cite[Lem.~4.3]{trees}}]
\label{cor:tree_spread}
The family $\calT_n$ is $(n/2, n-1)$-spread.
\end{corollary}

A forest $T$ is said to be \vocab{$c$-star-like}
if there exists an edge adjacent to at least $n/c$
other edges.
This implies that at least one vertex has 
degree at least $n/(2c)$.
The next goal of this subsection is to prove the 
following lemma.
\begin{lemma}
\label{lem:counting_avoiding_trees}
Suppose that $n\geq 1000$.
Let $F$ be a union of $t$-disjoint edges in $K_n$,
and let $T_0$ be a forest in $K_n$.
Let $N$ be the number of trees that contains $F$
but does not intersect $T_0\setminus F$.
\begin{enumerate}[label=(\alph*)]
\item If $T_0$ is not $12$-star-like, then 
$N \geq 0.01 |\calT_n(F)|$.
\item If $T_0$ is $12$-star-like but not a star, then 
$N\geq n^{n-t-100}$.
\end{enumerate}
\end{lemma}

This lemma was proven as \cite[Prop.~3.1]{trees}
using the Lopsided Lov\'asz Local Lemma,
but their proof contains an error in their Corollary 3.4.
We fix their proof.
We begin by reviewing the Lopsided Lov\'asz Local Lemma in \Cref{subsec:llll}.
Then, we prove \Cref{lem:counting_avoiding_trees} in 
\Cref{subsec:counting_trees}.
\subsection{Lopsided Lov\'asz Local Lemma} 
\label{subsec:llll}
The key tool that we will use is the Lopsided 
Lov\'asz Local Lemma, first introduced by Erd\H os and Spencer 
\cite{lopsided_lll}.
We first recall the statement of the lemma below.

Let $A_1,\dots,A_n$ be a collection of events 
in some probability space.
A \vocab{negative dependency graph} $G$ 
is a graph on vertex set $[n]$ such that for all $i\in [n]$,
$$\mathbb P\left(A_i \left| 
    \bigwedge_{j,\ (j,i)\in E(G)} \overline{A_j}
\right.\right) 
\leq \mathbb P(A_i).$$
\begin{lemma}[Lopsided Lov\'asz Local Lemma, LLLL]
\label{lem:llll}
Let $A_1,\dots,A_n$ be a collection of events 
with negative dependency graph $G$.
Suppose that there exists real numbers $x_1,\dots,x_n>0$
such that 
$$\mathbb P(A_i) \leq x_i \prod_{j,\ (j,i)\in E(G)}
(1-x_j).$$
Then, the probability that none of the events $A_1,\dots,A_n$
holds is at least $\prod_{j=1}^n (1-x_j)$.
\end{lemma}
\begin{proof}
The proof is exactly the same as the standard 
Lov\'asz Local Lemma: see \cite[Lem.~5.1.1]{alon_spencer}.
\end{proof}

Using \Cref{lem:counting_trees}, we now give 
the negative dependency graph that we will use.
We make the following setup.
\begin{itemize}
\item Let $F$ be a \textbf{spanning} forest in $K_n$.
\item Let $T$ be uniformly sampled from $\calT_n[F]$.
\item For each (not necessarily spanning) forest $H$, let $A_H$ be the event that 
$T$ contains $H$.
\end{itemize}
Two (not necessarily spanning) 
forests $H$ and $H'$ of $K_n$ are \vocab{$F$-disjoint}
if and only if there is no vertex 
$v$ of $H$ and $v'$ of $H'$ that are in the 
same connected component of $F$.

\begin{lemma}
\label{lem:forest_independent}
If $H_1$ and $H_2$ are $F$-disjoint, then 
$\mathbb P(A_{H_1}\wedge A_{H_2}) = 
\mathbb P(A_{H_1}) \mathbb P(A_{H_2})$.
\end{lemma}
\begin{proof}
We first note that if two vertices of $H_i$
are in the same connected component of $H_i$, then 
$F\cup H_i$ is a cycle, so both sides of the equality are 
immediately $0$, and so the lemma holds.
Otherwise, assume that for all $i$, all vertices of $H_i$ 
are in different connected component of $F$.
Combining with that $H_1,H_2$ are $F$-disjoint,
we deduce that all vertices in $H_1$ and $H_2$
are in different connected component of $F$.

Let $C_1,\dots,C_k$ be the connected component 
that vertices of $H_1$ are in.
Let $D_1,\dots,D_\ell$ be the connected component 
that vertices of $H_2$ are in.
Then, we have 
\begin{align*}
\mathbb P(A_{H_1}) 
&= \frac{|\calT_n[F\cup H_1]|}{|\calT_n[F]|}
= \frac{|C_1|+\dots+|C_k|}
{n^{k}\cdot |C_1|\cdots|C_k|} \\
\mathbb P(A_{H_2}) 
&= \frac{|\calT_n[F\cup H_2]|}{|\calT_n[F]|} 
= \frac{|D_1|+\dots+|D_\ell|}
{n^{\ell}\cdot |D_1|\cdots|D_\ell|},
\end{align*}
since changing from $F$ to $F\cup H_1$ merges 
components $C_1,\dots,C_k$ into one
and the same goes for $F\cup H_2$.
By a similar logic, we have 
$$\mathbb P(A_{H_1}\wedge A_{H_2})
= \frac{|\calT_n[F\cup H_1\cup H_2]|}{|\calT_n[F]|} 
= \frac{(|C_1|+\dots+|C_k|)(|D_1|+\dots+|D_\ell|)}
{n^{k+\ell}\cdot |C_1|\cdots|C_k|\cdot |D_1|\cdots |D_\ell|},$$
as desired.
\end{proof}
\begin{lemma}[Generalization of {\cite[Thm.~4]{lll_trees}}]
\label{lem:dep_graph}
Let $F$ be a spanning forest in $K_n$,
and let $T$ be uniformly sampled from $\calT_n[F]$.
Suppose that $H_1,\dots,H_k$ are (not necessarily spanning) 
forests of $K_n$.
Let $G$ be the graph on vertex set $[k]$ and 
$$(i,j)\in E(G) \text{ if and only if }
H_i\text{ and }H_j\text{ are } F\text{-disjoint}.$$
Then, $G$ is a negative dependency graph
for events $A_{H_1},\dots,A_{H_k}$.
\end{lemma}
\begin{proof}
Fix $i\in [k]$, and let $I$ be any subset of
$\{j : (j,i)\in E(G)\}$.
From the definition of conditional probability, we have to show that 
$$\mathbb P\left(A_{H_i}\bigwedge_{j\in I} \overline{A_{H_j}}\right)
\leq \mathbb P(A_{H_i}) \cdot \mathbb P\left(\bigwedge_{j\in I}\overline{A_{H_j}}\right).$$
In fact, we claim an equality.
To prove this, we apply inclusion-exclusion to obtain that
\begin{align*}
\mathbb P\left(A_{H_i}\wedge \bigwedge_{j\in I} 
\overline{A_{H_j}}\right)
&= \sum_{J\subseteq I} (-1)^{|J|}\cdot 
\mathbb P\left(A_{H_i}\wedge \bigwedge_{j\in J} 
A_{H_j}\right) 
\tag{inclusion-exclusion} \\
&= \sum_{J\subseteq I}
(-1)^{|J|}\cdot \bbP\left(A_{H_i} 
\wedge A_{\bigcup_{j\in J}H_j}\right)\\
&= \sum_{J\subseteq I}
(-1)^{|J|}\cdot \bbP\left(A_{H_i}\right)\cdot
\bbP\left(A_{\bigcup_{j\in J}H_j}\right)
\tag{\Cref{lem:forest_independent}} \\
&= \bbP(A_{H_i})
\sum_{J\subseteq N}
(-1)^{|J|} \bbP\left(
\bigwedge_{j\in J}A_{H_j}\right) \\
&= \bbP(A_{H_i}) 
\mathbb P\left(\bigwedge_{j\in I}\overline{A_{H_j}}\right),
\tag{inclusion-exclusion}
\end{align*}
as desired.
\end{proof}
To bound the probability, we note the following:
\begin{corollary}
\label{cor:prob_edge}
For any edge $e$, we have $\mathbb P(A_e) \leq \tfrac 2n$.
\end{corollary}
\begin{proof}
Follows immediately from \Cref{cor:tree_spread}.
\end{proof}
\subsection{Proof of \texorpdfstring
    {\Cref{lem:counting_avoiding_trees}}{Lemma \ref{lem:counting_avoiding_trees}}}
\label{subsec:counting_trees}
\begin{proof}[Proof of \Cref{lem:counting_avoiding_trees} (a)]
Let $T$ be tree uniformly sampled from $\calT_n[F]$.
Let $e_1,\dots,e_m$ be all edges in $T_0\setminus F$.
For each edge $e\in T_0\setminus F$,
we let $A_e$ denote the event that $T$ contains $e$.
We take the dependency graph $G$ as given by \Cref{lem:dep_graph}.
Note that since $T$ is not $12$-star-like,
each vertex has degree at most $n/12$,
so in the dependency graph, each vertex has degree at most $n/6$.

We pick $x_1=x_2=\dots=x_m=\tfrac 4n$. 
To show that this satisfies the hypothesis of 
the LLLL (\Cref{lem:llll}), we have to verify that 
$$\bbP(A_{e_i}) \leq \frac 4n \left(1-\frac 4n\right)^{n/6}.$$
From \Cref{cor:prob_edge}, we have to show that 
$\left(1-\tfrac 4n\right)^{n/6} \geq \tfrac 12$,
which can be easily checked to hold for all $n\geq 5$.
Thus, by the LLLL (\Cref{lem:llll}), we get that 
$$\frac{N}{|\calT_n[F]|} =
\mathbb P\left(\bigwedge_{i=1}^m \overline{A_{e_i}}\right)
\geq \prod_{i=1}^m (1-x_i) 
= \left(1-\tfrac 4n\right)^m \geq e^{-4} > 0.01,$$
implying the conclusion.
\end{proof}

\begin{proof}[Proof of \Cref{lem:counting_avoiding_trees} (b)]
The proof is essentially the same as 
\cite[Prop.~3.1]{trees}, except that the constant $6$
in their paper is replaced by the constant $12$
everywhere.

For simplicity, assume that $n$ even. 
(otherwise, remove a vertex and use the proof below for $n-1$.)
Initialize $F_0=F$ and $N_0=N$.
Let $v_0$ be the vertex with largest degree in $T_0$,
which has degree at least $n/24$.
We define $F_0'$ as follows.
\begin{itemize}
\item If $v_0$ is an isolated vertex in $F_0$,
then take another isolated vertex $u_0$, and let 
$F_0' = F_0\cup \{u_0v_0\}$.
This is possible because of parity.
\item Otherwise, let $F_0'=F$.
\end{itemize}
Then, we let 
$$T_1 = T_0\setminus \{v_0\},\quad F_1=F_0\setminus\{v_0\}.$$
Let $N_1$ be the number of trees containing $F_1$
and do not meet $T_1\setminus F_1$.
We claim that $N_1\leq N_0$.
To prove this, we note that any tree $T$ containing $F_1$
and does not meet $T_1\setminus F_1$
can be converted to a tree $T'$ containing $F_0$
and does not meet $T_0\setminus F_0$
by adding edge $u_0v_0$.

If $T_1$ is not $12$-star-like, then we may apply 
\Cref{lem:counting_avoiding_trees} (a).
We now repeat the process: take vertex $v_1$ with largest 
degree in $T_1$.
Then, define $T_2=T_1\setminus\{v_1\}$ and $F_2=F_1\setminus\{v_1\}$,
and repeat.

We claim that the process must stop after defining
$T_{24}$ and $F_{24}$.
Assume not. Then, the number of remaining edges 
in $T_{24}$ is at most 
$$n-1 - \left(\frac{n}{24} + \frac{n-1}{24} + 
\dots + \frac{n-23}{24}\right) \leq 12.5,$$
which is less than $\tfrac{n-24}{24}$ if $n>1000$,
a contradiction. 
Note also that $|F_{24}|\leq t$.
We then apply (a) to find that 
$$N\geq N_{24} \geq 0.01|\calT_{n-24}[F_{24}]| 
= 0.01 \cdot 2^t(n-24)^{n-t-26} > n^{n-t-100}$$
when $n>1000$, as desired.
\end{proof}
\subsection{Enumerating Star-like Trees}
We make the following useful observation.
\begin{lemma}
\label{lem:few_star_like_trees}
There are at most $2^nn^{n-\frac n{2c}}$ spanning trees of $K_n$
that are $c$-star-like.
\end{lemma}
\begin{proof}
Number the vertices of $K_n$ by $1,2,\dots,n$.
Recall the following generalization 
of Cayley's tree formula (see, e.g., \cite[(2.4)]{tree_formula}
for the proof):
$$\sum_{T\in\calT_n} x_1^{\deg_T(1)} 
\cdots x_n^{\deg_T(n)}
= x_1x_2\cdots x_n(x_1+x_2+\dots+x_n)^{n-2}.$$
We now show that there are at most $2^nn^{n-n/(2c)-1}$ trees $T$
for which $\deg_T(1) \geq \tfrac n{2c}$.
To see this, plug in $x_1=n$ and $x_2=\dots=x_n=1$
in the above identity to get that there are at most
$$\frac{n(n+(n-1))^{n-2}}{n^{n/(2c)}} 
\leq 2^n n^{n-n/(2c)-1}$$
such trees. The result then follows.
\end{proof}
\section{Putting Everything Together}
\label{sec:main_proof}
\begin{proof}[Proof of \Cref{thm:main}]
We let $N=\binom n2$ and $\calU = \calT_n\subseteq 
\binom{[N]}{n-1}$ be the family of 
all spanning trees with $n$ vertices,
which is $(n/2, n-1)$-spread by \Cref{cor:tree_spread}.

Take $\eps=1$ in \Cref{lem:strong_spread_approx},
and select the corresponding $\delta$.
We assume that $n$ is sufficiently large,
and also assume that $t\geq n^{1-\delta}/\delta$.
(Otherwise, we can use the previous result, 
\cite{trees}).
Then, if $c<\tfrac 1{20}$, then all hypotheses of 
\Cref{lem:strong_spread_approx} is satisfied.
Thus, there exists $\calS\subseteq \binom{[N]}{\leq 2t}$
such that $\calS$ is $t$-intersecting
and $\calA' = \calA\setminus\calA[\calS]$ 
satisfies $|\calA'| \leq 0.001c_t(\calT_n)
= 0.001 \cdot 2^t n^{n-t-2}$ (by \Cref{cor:trivial_case}).

Assume for a contradiction that $|\calA|\geq 2^tn^{n-t-2}$.
We have two cases.
\begin{itemize}
\item \textbf{If $\calS$ is nontrivial},
then we apply \Cref{thm:structure} on $\calS$ 
with $\eps=\tfrac 12$
(in particular, $q=2t$ and $r=n/2$).
Thus, if $c<\tfrac 1{96}$, then
there exists a forest $F$ with $t$ edges such that 
$|\calA[\calS]| \leq \tfrac 12|\calA[F]|$.
Then, we have that 
\begin{align*}
|\calA| &\leq |\calA[\calS]| + |\calA'| \\
&\leq \tfrac 12|\calA[F]| + 0.001\cdot 2^tn^{n-t-2} \\
&\leq \tfrac 12 \cdot 2^t n^{n-t-2} + 0.001\cdot 2^tn^{n-t-2}
\tag{by \Cref{cor:trivial_case}} \\
&< 2^t n^{n-t-2},
\end{align*}
which is a contradiction.
\item \textbf{If $\calS$ is trivial,}
then suppose that every element in $\calS$ 
contains a forest $F$ that has $t$ edges.
Thus, we have $|\calA[S]| \leq |\calA[F]|$.

If $F$ is not a disjoint union of $t$ edges,
then \Cref{cor:trivial_case} (b) gives 
$|\calA[F]| \leq \tfrac 3{4} 2^t n^{n-t-2}$. Thus,
$$|\calA| \leq |\calA[F]| + |\calA'|
\leq 0.76 \cdot 2^t n^{n-t-2},$$
which is a contradiction.

Otherwise, $F$ is a disjoint union of $t$ edges.
If $\calA'=\emptyset$, then $\calA\subseteq\calA[F]$,
and so $|\calA|\leq 2^tn^{n-t-2}$.
Otherwise, assume that $\calA'\neq\emptyset$.
We have three cases.
\begin{itemize}
\item \textbf{If there exists $T_0\in \calA'$ 
that is not $24$-star-like,}
then every tree $T$ that contains $F$
but $T\cap (T_0\setminus F) = \emptyset$
must not be in $\calA$.
By \Cref{lem:counting_avoiding_trees} (a), there are at least 
$0.01|\calT_n[F]|$ such trees, so we have 
\begin{align*}
|\calA| &\leq |\calT_n[F]| - 0.01|\calT_n[F]| + |\calA'| \\
&\leq 0.99\cdot 2^t n^{n-t-2} + 0.01 n^{n-t-2} \\
&< 2^t n^{n-t-2},
\end{align*}
which is a contradiction.
\item \textbf{If every tree in $\calA'$
is $24$-star-like and there exists 
$T_0\in\calA'$ that is not a star,}
then by \Cref{lem:few_star_like_trees},
we get that $|\calA'| \leq 2^n n^{23n/24+1}
< n^{0.98n}$ for all large $n$.

Furthermore, same as previous case,
every tree that contains $F$ but does not meet $T_0\setminus F$
is not in $\calA$.
By \Cref{lem:counting_avoiding_trees}, there are at least 
$n^{n-t-51}$ such trees.
Thus, if $c<\tfrac{1}{100}$, then we have 
\begin{align*}
|\calA| &\leq |\calT_n[F]| - n^{n-t-51} + |\calA'| \\
&\leq 2^tn^{n-t-2} - n^{n-t-100} + n^{0.98n} \\
&< 2^tn^{n-t-2}
\end{align*}
which is a contradiction.
\item \textbf{If every tree in $\calA'$ is a star,}
then since $t\geq 2$, no two stars are $t$-intersecting, 
so assume that $\calA' = \{S\}$ for some star $S$.
Let $v$ be the (unique) non-leaf vertex of $S$.
We claim that there are at least two trees
that contain $F$ and does not meet $F\setminus S$.
To prove this, we split into two cases.
\begin{itemize}
\item \textbf{If $v$ is adjacent to an edge in $F$,}
then at least two vertices in $F$ are not $v$.
Take one such vertex $w$.
Then, complete $F$ into a tree by only adding edges 
incident to $w$. 
Since there are at least two choices of $w$,
we can construct at least two trees.
\item \textbf{If $v$ is an isolated vertex in $F$,}
then any tree containing $F$ and has $v$ as a leaf will work.
\end{itemize}
Therefore $|\calT_n[F]\setminus\calA|\geq 2$,
which implies that $|\calA| \leq |\calT_n[F]|-1 
< 2^tn^{n-t-2}$, a contradiction. \qedhere
\end{itemize}
\end{itemize}
\end{proof}
\printbibliography
\end{document}